\documentclass[12pt,leqno]{amsart}  

\usepackage{calc}
\usepackage[left=1in,right=1in,top=1in,bottom=1in]{geometry}  
\usepackage{graphicx}
\usepackage{amsmath,amssymb,epsfig}
\usepackage{amsthm}
\usepackage{enumerate}

\newcommand{\N}{{\mathbb N}}
\newcommand{\R}{{\mathbb R}}
\newcommand{\Z}{{\mathbb Z}}
\newcommand{\Q}{{\mathbb Q}}
\newcommand{\T}{{\mathcal T}}
\newcommand{\A}{{\mathcal A}}

\newcommand{\Gal}{{\mathrm{Gal}}}
\newcommand{\Sel}{{\mathrm{Sel}}}
\newcommand{\dimF}{{\mathrm{dim}_{\mathbb{F}_2}}}
\newcommand{\Ftwo}{{\mathbb{F}_2}}
\newcommand{\F}{{\mathbb{F}}}

\newcommand{\Zt}{{\mathbb{Z}/2\mathbb{Z}}}

\newcommand{\hatphi}{{ \hat \phi }}

\newcommand{\p}{{ \mathcal{P} }}

\newcommand{\ord}{{ \mathrm{ord} }}

\newtheorem{thm}{\bf{Theorem}}
\newtheorem{theorem}{\bf{Theorem}}[section]

\newtheorem{proposition}[theorem]{\bf{Proposition}}

\newtheorem{lemma}[theorem]{\bf{Lemma}}

\theoremstyle{definition}
\newtheorem{definition}[theorem]{\bf{Definition}}

\input xy 
\xyoption{all}

\usepackage[OT2,T1]{fontenc}
\DeclareSymbolFont{cyrletters}{OT2}{wncyr}{m}{n}
\DeclareMathSymbol{\Sha}{\mathalpha}{cyrletters}{"58}

\begin{document}

\bibliographystyle{alpha}



\title[{\tiny Distribution of Selmer Ranks of Twists of Elliptic Curves with Partial Two-Torsion}]{On the Distribution of 2-Selmer Ranks within Quadratic Twist Families of Elliptic Curves with Partial Rational Two-Torsion}
\author{Zev Klagsbrun}

\begin{abstract}This paper presents a new result concerning the distribution of 2-Selmer ranks in the quadratic twist family of an elliptic curve with a single point of order two that does not have a cyclic 4-isogeny defined over its two-division field. We prove that at least half of all the quadratic twists of such an elliptic curve have arbitrarily large 2-Selmer rank, showing that the distribution of 2-Selmer ranks in the quadratic twist family of such an elliptic curve differs from the distribution of 2-Selmer ranks in the quadratic twist family of an elliptic curve having either no rational two-torsion or full rational two-torsion.
\end{abstract}

\maketitle

\pagenumbering{arabic}

\section{Introduction}

\subsection{Distributions of Selmer Ranks}

Let $E$ be an elliptic curve defined over $\Q$ and let $\Sel_2(E/\Q)$ be its 2-Selmer group (see Section \ref{bg} for its definition). We define the 2-Selmer rank of $E/\Q$, denoted $d_2(E/\Q)$, by $$d_2(E/\Q) =  \dimF  \Sel_2(E/\Q) - \dimF E(\Q)[2].$$ 

For a given elliptic curve and non-negative integer $r$, we are able to ask what proportion of the quadratic twists of $E$ have 2-Selmer rank equal to $r$.

Let $S(X)$ be the set of squarefree natural numbers less than or equal to $X$. Heath-Brown proved that for the congruent number curve $y^2 = x^3 - x$, there are explicits constants $\alpha_0, \alpha_1, \alpha_2, \ldots$ summing to one such that $$\lim_{X \rightarrow \infty}\frac{|\{ d  \in S(X)  : d_2(E^d/\Q) = r  \} |}{|S(X)|} = \alpha_r$$ for every $r \in  \Z^{\ge 0}$, where $E^d$ is the quadratic twist of $E$ by $d$ \cite{HB}. This result was extended by Swinnerton-Dyer and Kane to all elliptic curves $E$ over $\Q$ with $E(\Q)[2] \simeq \Zt \times \Zt$ that do not have a cyclic 4-isogeny defined over $\Q$ \cite{Kane}, \cite{SD}.  More recently, Klagsbrun, Mazur, and Rubin showed that a version of this result is true for curves $E$ with $E(\Q)[2] = 0$ and $\Gal(\Q(E[2])/\Q) \simeq \mathcal{S}_3$ when a different method of counting is used \cite{KMR}. These results state that if the mod-4 representation of a curve $E$ satisfies certain conditions, then there is a discrete distribution on 2-Selmer ranks within the quadratic twist family of $E$. We show that this is not the case when $E(\Q)[2] \simeq \Zt$ and $E$ does not have a cyclic 4-isogeny defined over $\Q(E[2])$. Specifically, we prove the following:

\begin{thm}\label{mainthm}
Let $E$ be an elliptic curve defined over $\Q$ with $E(\Q)[2] \simeq \Z/2\Z$ that does not have a cyclic isogeny defined over $\Q(E[2])$. Then for any fixed $r$, $$\liminf_{X\rightarrow \infty}\frac{ \left | \{ d \in S(X) : d_2(E^d/\Q) \ge r \} \right | }{|S(X)|}  \ge \frac{1}{2}$$ and $$\liminf_{X\rightarrow \infty}\frac{ \left | \{ \pm d \in S(X) : d_2(E^d/\Q) \ge r \} \right | }{2|S(X)|}  \ge \frac{1}{2}.$$
\end{thm}

In particular, this shows that there is not a distribution function on 2-Selmer ranks within the quadratic twist family of $E$.

Theorem \ref{mainthm} is proved by way of the result. 

\begin{thm}\label{Tconverge}  Let $E$ be an elliptic curve defined over $\Q$ with $E(\Q)[2] \simeq \Z/2\Z$ that does not have a cyclic isogeny defined over $\Q(E[2])$. Then the normalized distribution $$\displaystyle \frac{P_r(\T(E/E^\prime), X)}{\sqrt{\frac{1}{2} \log \log X}}$$ converges weakly to the Gaussian distribution $$G(z) = \frac{1}{\sqrt{2\pi}} \int_{-\infty}^{z} \! e^\frac{w^2}{2} \, \mathrm{d}w,$$ where $$P_r(\T(E/E^\prime), X) = \displaystyle{ \frac{ |\{d \in S(X) : \ord_2 \T(E^d/E^{\prime d}) \le r \}| }{|S(X)|} }$$ for $X \in \R^+$, $r \in \Z^{\ge 0}$, and $\T(E^d/E^{\prime d})$ as defined in Section \ref{bg}.

\end{thm}

Theorem \ref{Tconverge} will follow from a variant of the Erd{\"o}s-Kac theorem for squarefree numbers which is proved in Appendix \ref{apf}.

Xiong and Zaharescu proved results similar to Theorems \ref{mainthm} and \ref{Tconverge} in the special case when $E(\Q)[2] \simeq \Z/2\Z$ and $E$ has a cyclic 4-isogeny defined over $\Q$ \cite{XZ}.

\subsection{Layout}
We begin in Section \ref{bg} by recalling the definitions of the 2-Selmer group and the Selmer groups associated with a 2-isogeny $\phi$ and presenting some of the connections between them. Section \ref{twistplace} examines the behavior of the local conditions for the $\phi$-Selmer groups under quadratic twist. We prove Theorems \ref{mainthm} and \ref{Tconverge} in Section \ref{pfofmain} by combining the results of Sections \ref{bg} and \ref{twistplace} with a variant of the Erd{\"o}s-Kac theorem for squarefree numbers which we prove in Appendix \ref{apf}.

\subsection{Acknowledgements}
I would like to express my thanks to Karl Rubin for his helpful comments and suggestions, to Ken Kramer for a series of valuable discussions, as well as to Michael Rael and Josiah Sugarman for helpful conversations regarding the Erd{\"o}s-Kac theorem.

\section{Selmer Groups}\label{bg}

We begin by recalling the definition of the 2-Selmer group. If $E$ is an elliptic curve defined over a field $K$, then $E(K)/2(K)$ maps into $H^1(K, E[2])$ via the Kummer map. The following diagram commutes for every place $p$ of $\Q$, where $\delta$ is the Kummer map.

\begin{center}\leavevmode
\begin{xy} \xymatrix{
E(\Q)/2E(\Q)  \ar[d] \ar[r]^{\delta} & H^1(\Q, E[2]) \ar[d]^{Res_p} \\
E(\Q_p)/2E(\Q_p)   \ar[r]^{\delta} & H^1(\Q_p, E[2]) }
\end{xy}\end{center}

For each place $p$ of $\Q$, we define a distinguished local subgroup $H^1_f(\Q_p, E[2]) \subset H^1(\Q_p, E[2])$ by $\text{Image} \left (\delta: E(\Q_p)/2E(\Q_p) \hookrightarrow H^1(\Q_p, E[2]) \right )$. We define the \textbf{2-Selmer group} of $E/\Q$, denoted $\Sel_2(E/\Q)$, by $$\Sel_2(E/\Q) = \ker \left ( H^1(\Q, E[2]) \xrightarrow{\sum res_p} \bigoplus_{p\text{ of } \Q} H^1(\Q_p, E[2])/H^1_f(\Q_p, E[2]) \right ).$$

The  $2$-Selmer group is a finite dimensional $\F_2$-vector space that sits inside the exact sequence of $\F_2$-vector spaces $$0\rightarrow E(\Q)/2E(\Q) \rightarrow \Sel_2(E/\Q) \rightarrow \Sha(E/\Q)[2] \rightarrow 0$$ where $\Sha(E/\Q)$ is the Tate-Shafaravich group of $E$.

\begin{definition}
We define the \textbf{2-Selmer rank of $E$}, denoted $d_2(E/\Q)$, by $$d_2(E/\Q) =  \dimF   \Sel_2(E/\Q) - \dimF  E(\Q)[2].$$
\end{definition}

If $E(\Q)$ has a point of order two, then we can define a Selmer group arising from the two-isogeny with kernel generated by that point. Explicitly, if $E$ is an elliptic curve defined over $\Q$ with $E(\Q)[2] \simeq \Zt$, then $E$ can be given by a model $y^2 = x^3 + Ax^2 + Bx$ with $A, B \in \Z$. The subgroup  $C = E(\Q)[2]$ is then generated by the point $P = (0,0)$.

Given this model, we are able to define an isogeny $\phi:E \rightarrow E^\prime$ with kernel $C$, where $E^\prime$ is given by the model $y^2 = x^3 - 2Ax^2 + (A^2 - 4B)x$ and $\phi$ is given by $\phi(x, y) = \left( \left (\frac{x}{y} \right)^2, \frac{y(B-x^2)}{x^2} \right )$ for $(x, y) \not \in C$. The isogeny $\phi$ gives rise to a pair of Selmer groups.

The isogeny $\phi$ gives a short exact sequence of $G_\Q$ modules \begin{equation} 0 \rightarrow C \rightarrow E(\overline{\Q}) \xrightarrow{\phi}  E^\prime(\overline{\Q}) \rightarrow 0.\end{equation} This sequence gives rise to a long exact sequence of cohomology groups $$0 \rightarrow C \rightarrow E(\Q) \xrightarrow{\phi} E^\prime(\Q) \xrightarrow{\delta} H^1(\Q, C) \rightarrow H^1(\Q, E) \rightarrow H^1(\Q, E^\prime) \ldots$$ The map $\delta$ is given by $\delta(Q)(\sigma)= \sigma(R) - R$ where $R$ is any point on $E$ with $\phi(R) = Q$. 

This sequence remains exact when we replace $\Q$ by its completion $\Q_p$ at any place $p$, which gives rise to the following commutative diagram.

\begin{center}\leavevmode
\begin{xy} \xymatrix{
E^\prime(\Q)/\phi(E(\Q))  \ar[d] \ar[r]^\delta & H^1(\Q, C) \ar[d]^{Res_p} \\
E^\prime(\Q_p)/\phi(E(\Q_p))   \ar[r]^\delta & H^1(\Q_p, C) }
\end{xy}\end{center}

In a manner similar to how we defined the 2-Selmer group, we define distinguished local subgroups $H^1_\phi(\Q_p, C)\subset  H^1(\Q_p, C)$ as the image of $E^\prime(\Q_p)/\phi(E(\Q_p))$ under $\delta$ for each place $p$ of $\Q$.  We define the \textbf{$\mathbf \phi$-Selmer group of $\mathbf E$}, denoted $\Sel_\phi(E/\Q)$ as  $$\Sel_\phi(E/\Q) = \ker \left ( H^1(\Q, C) \xrightarrow{\sum res_p} \bigoplus_{p \text{ of } \Q} H^1(\Q_p, C)/H^1_\phi(\Q_p, C) \right ).$$ The group $\Sel_\phi(E/\Q)$ is a finite dimensional $\Ftwo$-vector space and we denote its dimension $\dimF \Sel_\phi(E/\Q)$ by $d_\phi(E/\Q)$.

The isogeny $\phi$ on $E$ gives gives rise to a dual isogeny $\hat \phi$ on $E^\prime$ with kernel $C^\prime = \phi(E[2])$. Exchanging the roles of $(E, C, \phi)$ and $(E^\prime, C^\prime, \hat \phi)$ in the above defines the $\mathbf{\hat \phi}$\textbf{-Selmer group}, $\Sel_\hatphi(E^\prime/\Q)$, as a subgroup of $H^1(\Q, C^\prime)$. The following two theorems allow us to compare the $\phi$-Selmer group, the $\hat \phi$-Selmer group, and the 2-Selmer group .

\begin{theorem}\label{gss}The $\phi$-Selmer group, the $\hat \phi$-Selmer group, and the 2-Selmer group sit inside the exact sequence \begin{equation}0 \rightarrow E^\prime(\Q)[2]/\phi(E(\Q)[2]) \rightarrow \Sel_\phi(E/\Q) \rightarrow \Sel_2(E/\Q) \xrightarrow{\phi}\Sel_\hatphi(E^\prime/\Q).\end{equation}
\end{theorem}
\begin{proof}
This is a well known diagram chase. See Lemma 2 in \cite{FG} for example.
\end{proof}

The Tamagawa ratio $\T(E/E^\prime)$  gives a second relationship between $d_\phi(E/\Q)$ and $d_{\hat \phi}(E^\prime/\Q)$.

\begin{definition} The ratio $$\mathcal{T}(E/E^\prime) = \frac{ \big | \Sel_\phi(E/\Q)  \big |}{\big |\Sel_{\hat \phi}(E^\prime/\Q)\big |}$$ is called the \textbf{Tamagawa ratio} of $E$.
\end{definition}

\begin{theorem}[Cassels]\label{prodform2}
The Tamagawa ratio $\mathcal{T}(E/E^\prime)$ is given by $$\mathcal{T}(E/E^\prime) = \prod_{p}\frac{\left | H^1_\phi(\Q_p, C)\right |}{2}.$$
\end{theorem}
\begin{proof}
This is a combination of Theorem 1.1 and equations (1.22) and (3.4) in \cite{Cassels8}.
\end{proof}

Stepping back, we observe that if the Tamagawa ratio $\T(E/E^\prime) \ge 2^{r+2}$ , then $d_\phi(E/\Q) \ge r+2$, and therefore by Theorem \ref{gss}, $d_2(E/\Q) \ge r$. (If $E$ does not have a cyclic 4-isogeny defined over $\Q$ then we can in fact show that $\T(E/E^\prime) \ge 2^r$ implies that $d_2(E/\Q) \ge r$, but this is entirely unnecessary for our purposes.)

\section{Local Conditions at Twisted Places}\label{twistplace}

For the remainder of this paper, we will let $E$ be an elliptic curve with $E(\Q)[2] \simeq \Zt$ that does not have a cyclic 4-isogeny defined over $\Q(E[2])$ and let $\phi$ be the isogeny with kernel $C=E(\Q)[2]$ defined in Section \ref{bg}.

If $p \nmid 2\infty$ is prime where $E$ has good reduction, then $H^1_\phi(\Q_p, C)$ is a 1-dimensional $\F_2$ subspace of $H^1(\Q_p, C)$ equal to the unramified local subgroup $H^1_u(\Q_p, C)$. If $E$ has good reduction at $p \nmid 2$ and $p \mid d$, then the twisted curve $E^d$ will have bad reduction at $d$. The following lemma addresses the size of $H^1_\phi(\Q_p, C^d)$.

\begin{lemma}\label{localconds}
Suppose $p \ne 2$ is a prime where $E$ has good reduction and $d \in \Z$ is squarefree with $p \mid d$.

\begin{enumerate}[(i)]
\item If  $E(\Q_p)[2] \simeq \Zt \simeq E^\prime(\Q_p)[2]$, then $\dimF H^1_\phi(\Q_p, C^d) = 1$.
\item If  $E(\Q_p)[2] \simeq \Zt \times \Zt \simeq E^\prime(\Q_p)[2]$, then $\dimF H^1_\phi(\Q_p, C^d) = 1$.
\item If $E(\Q_p)[2] \simeq \Zt$ and $E^\prime(\Q_p)[2] \simeq \Zt \times \Zt$, then $\dimF H^1_\phi(\Q_p, C^d) = 2$.
\item If $E(\Q_p)[2] \simeq \Zt \times \Zt$ and $E^\prime(\Q_p)[2] \simeq \Zt$, then $\dimF H^1_\phi(\Q_p, C^d) = 0$.
\end{enumerate}
\end{lemma}
\begin{proof}
Lemma 3.7 in \cite{K} shows that $E^{\prime d}(\Q_p)[2^\infty]/\phi( E^d(\Q_p)[2^\infty]) = E^{\prime d}(\Q_p)[2]/\phi( E^d(\Q_p)[2])$. All four results then follow immediately.
\end{proof}

The following proposition suggests that each of the cases in Lemma \ref{localconds} should occur equally often.
\begin{proposition}\label{charac}
If $E$ is an elliptic curve with $E(\Q)[2] \simeq \Zt$ that does not have a cyclic 4-isogeny defined over $\Q(E[2])$, then $E^\prime(\Q)[2]\simeq \Zt$ and $\Q(E[2]) \ne \Q(E^\prime[2])$.
\end{proposition}
\begin{proof}
Let $Q^\prime \in E^\prime[2] - C^\prime$, $C = \langle P \rangle$, and take $Q \in E[4]$ with $\phi(Q) = Q^\prime$. Since $Q^\prime \in E^\prime(\Q)[2] - C^\prime$, and both $\phi \circ \hatphi = [2]_{E^\prime}$ and $\hatphi \circ \phi = [2]_{E}$, it follows that $2Q = \hatphi(Q^\prime) = P.$ Let $M = \Q(E[2])$. Since $E$ has no cyclic 4-isogeny defined over $M$, there exists $\sigma \in G_M$ such that $\sigma(Q) \not \in \langle Q \rangle = \left \{0 , Q, P, Q + P \right \}.$ In particular, since $\phi^{-1}(Q^\prime) \subset \langle Q \rangle$, we get that $\phi(\sigma(Q)) \ne Q^\prime$. We then get that $$\sigma(Q^\prime) = \sigma(\phi(Q)) = \phi(\sigma(Q)) \ne Q^\prime,$$ showing that $Q^\prime$ is not defined over $M$, and therefore that $\Q(E^\prime[2]) \not \subset M$. It then follows that $\Q(E[2])$ and $\Q(E^\prime[2])$ are disjoint quadratic extensions of $\Q$ and that $E^\prime(\Q)[2] \simeq \Zt$.
\end{proof}

\section{Proof of Main Theorems}\label{pfofmain}
In this section we prove Theorems \ref{mainthm} and \ref{Tconverge} by analyzing the behavior of $\T(E/E^\prime)$ under quadratic twist and employing a variant of the Erd{\"o}s-Kac theorem.

We begin by recalling the following definition.

\begin{definition}
A function $g:\N \rightarrow \R$ is called \textbf{additive} if $g(nm) = g(n) + g(m)$ whenever $n$ and $m$ are relatively prime.
\end{definition}

If $g(n)$ is an additive function, then the classical Erd{\"o}s-Kac theorem gives the distribution of $g(n)$ under mild hypothesis. The following variant of the Erd{\"o}s-Kac theorem is for additive functions defined on the set of squarefree numbers $S$.

\begin{theorem}\label{EKSF}
Let $S$ be the set of squarefree natural numbers and suppose that $g:S \rightarrow \R$ is an  additive function such that $|g(p)| \le 1$ for all primes p. Let $$A(x) = \sum_{p \le x} \frac{g(p)}{p} \text{ and } B(x) = \sqrt {\sum_{p \le x} \frac{g(p)^2}{p}},$$ where the sums are taken over all primes less than or equal to $x$. If $B(x) = \omega \left( \log \log\log(x) \right )$ (in the asymptotic sense), then $$\displaystyle {v_x(n; g(n) - A(x) \le zB(x) ) = \frac{|\{n \in S(x) : g(n) - A(x) \le zB(x) \}|}{|S(x) |}} $$ converges weakly (i.e. pointwise in $z$) to $$G(z) = \frac{1}{\sqrt{2\pi}} \int_{-\infty}^{z} \! e^\frac{w^2}{2} \, \mathrm{d}w.$$
\end{theorem}

\begin{proof}
See Appendix \ref{apf}.
\end{proof}

For $n \in S$, define an additive function $g(n)$ by $$g(n) = \sum_{\genfrac{}{}{0pt}{}{p \mid n}{p \nmid 2\Delta \infty}} \frac{ \left ( \frac{\Delta^\prime}{p} \right ) - \left(  \frac{\Delta}{p}  \right ) }{2}$$ where $\Delta$ is the discriminant of some integral model of $E$ and $\Delta^\prime$ is the discriminant of some integral model of $E^\prime$. That is, $g(d)$ roughly counts the difference between the number of primes dividing $d$ where condition $(iii)$ of Proposition \ref{localconds} is satisfied and the number of primes dividing $d$ where condition $(iv)$ is satisfied. The value $g(d)$ can therefore be connected to the Tamagawa ratio $\T(E^d/E^{\prime d})$ in the following manner.

\begin{proposition}\label{TT}
The order of $2$ in the Tamagawa ratio $\T(E^d/E^{\prime d})$ is given by $$\ord_2 \T(E^d/E^{\prime d}) = g(d) + \sum_{v |2\Delta\infty} \left ( \dimF H^1_\phi(\Q_v, C^d)  - 1 \right ).$$

\end{proposition}
\begin{proof}
By Thereom \ref{prodform2}, $\ord_2 \T(E^d/E^{\prime d})$ is given by $$\displaystyle{ \ord_2 \T(E^d/E^{\prime d}) = \sum_{v |2d\Delta\infty} \left ( \dimF H^1_\phi(\Q_v, C^d)  - 1 \right ).}$$ Since $E(\Q_p)[2] \simeq \Zt \times \Zt$ if and only $\left ( \frac{\Delta}{p} \right )= 1$ and $E^\prime(\Q_p)[2] \simeq \Zt \times \Zt$ if and only $\left ( \frac{\Delta^\prime}{p} \right )= 1$ for primes $p \mid d$ with $p\nmid 2\Delta$, Lemma \ref{localconds} gives us that $$\dimF H^1_\phi(\Q_p, C^d) - 1 =  \frac{ \left ( \frac{\Delta^\prime}{p} \right ) - \left(  \frac{\Delta}{p}  \right ) }{2} $$ for places $p \mid d$ with $p\nmid 2\Delta\infty$.
\end{proof}

The following proposition will allow us to evaluate $A(x)$ and $B(x)$ for $g(n)$.

\begin{proposition}\label{mertenS}
Let $c$ be a non-square integer. Then $$\sum_{p \le x}\frac{1 +\left (\frac{c}{p} \right ) }{p} = \log \log x + O(1).$$
\end{proposition}
\begin{proof}
This is an application of Lemma 2.11 in \cite{Elliot1}.
\end{proof}
\begin{proof}[Proof of Theorem \ref{Tconverge}]
We wish to apply Theorem \ref{EKSF} to $g(n)$. We may rewrite $A(x)$ as $$\frac{1}{2}\sum_{\genfrac{}{}{0pt}{}{\ord_p n\text{ odd}}{p \nmid 2\Delta \infty}}\frac{1 +\left (\frac{\Delta^\prime}{p} \right ) }{p} - \frac{1}{2}\sum_{\genfrac{}{}{0pt}{}{\ord_p n\text{ odd}}{p \nmid 2\Delta \infty}}\frac{1 +\left (\frac{\Delta}{p} \right ) }{p}.$$ As $\Delta^\prime$ is not a square by Proposition \ref{charac}, we therefore get that $A(x) = O(1)$ by Proposition \ref{mertenS}. We can rewrite $B(x)$ as $$B(x) =\sqrt{  \displaystyle{ \sum_{\genfrac{}{}{0pt}{}{p \le x,  p \nmid 2\Delta\infty }{ \left (\frac{\Delta}{p} \right ) \ne \left ( \frac{\Delta^\prime}{p} \right )} }\frac{1}{p}}} =\sqrt{ \frac{1}{2}\displaystyle{ \sum_{\genfrac{}{}{0pt}{}{p \le x}{p \nmid 2\Delta\infty }}\frac{1 - \left (\frac{\Delta\Delta^\prime}{p} \right ) }{p}}}.$$ 
By Proposition \ref{charac}, $\Delta$ and $\Delta^\prime$ do not differ by a square, so therefore $B(x) = \displaystyle{\sqrt{ \frac{1}{2}\log \log x} + O(1)}$ by Proposition \ref{mertenS}. Applying Theorem \ref{EKSF} to $g(n)$, we then get that $$v_x \left ( n; \frac{g(n) }{\sqrt{\frac{1}{2} \log \log x}}\le z \right )$$ converges weakly to $G(z)$. By Proposition \ref{TT}, $g(d) = \ord_2 \T(E^d/E^{\prime d}) + O(1)$, so the result follows. 

\end{proof}

\begin{proof}[Proof of Theorem \ref{mainthm}]
By Theorem \ref{Tconverge}, $$\lim_{X \rightarrow \infty} \displaystyle{ \frac{ |\{d \in S(X) : \ord_2 \T(E^d/E^{\prime d}) \ge r \}| }{|S(X)|} } = \frac{1}{2}$$ for any fixed $r \ge 0$. As $d_2(E^d/\Q) \ge  \ord_2 \T(E^d/E^{\prime d}) - 2$, this shows that for any $\epsilon > 0$, $$\displaystyle{ \frac{ |\{d \in S(X) : d_2 (E^d/\Q) \ge r \}| }{|S(X)|} } \ge \frac{1}{2} - \epsilon$$ for sufficiently large $X$. As twisting $E$ by $-d$ is equivalent to twisting $E^{-1}$ by $d$, the remainder of the result follows.
\end{proof}

\bibliography{citations.bib}

\appendix

\section{An Erd{\"o}s-Kac Theorem for Squarefree Numbers}\label{apf}

The purpose of this appendix is to prove the following:
 
\noindent \textbf{Theorem \ref{EKSF}.} \textit{
Let $S$ be the set of squarefree natural numbers and suppose that $g:S \rightarrow \R$ is an  additive function such that $|g(p)| \le 1$ for all primes p. Let $$A(x) = \sum_{p \le x} \frac{g(p)}{p} \text{ and } B(x) = \sqrt {\sum_{p \le x} \frac{g(p)^2}{p}}.$$ If $B(x) = \omega \left( \log \log\log(x) \right )$ (in the asymptotic sense), then $$\displaystyle {v_x(n;g(n) - A(x) \le zB(x) ) = \frac{|\{n \in S(x) : f(n) - A(x) \le zB(x) \}|}{|S(x) |}} $$ converges weakly (i.e. pointwise in $z$) to $$G(z) = \frac{1}{\sqrt{2\pi}} \int_{-\infty}^{z} \! e^\frac{w^2}{2} \, \mathrm{d}w.$$
}

The proof we present here is based on the framework developed by Granville and Soundararajan in \cite{GS} which is outlined in the following section.

\subsection{Sieving and the Erd{\"o}s-Kac Theorem}\label{sEK}

Let $\A$ be a finite sequence of natural numbers $a_1, a_2, \ldots, a_N$, $\p$ a set of primes, and $g: \A \rightarrow \R$ an additive function supported on $\p$ such that $|g(p)| \le 1$ for every prime $p \in \p$. The goal is to identify the distribution of $g(n)$ on $\A$ and we pursue this by approximating its moments.

For any $d \in \A$, we define $\A_d = \{ n \in \A : d \mid n \}$. Suppose that there is some multiplicative function $h(n) : \A \rightarrow \R$ such that we may write $|\A_d|$ as $|\A_d| = \frac{h(d)}{d}N + r_d$ for every $d \in \A$, where we think of $\frac{h(d)}{d}$ as the approximate proportion of elements $\A$ divisible by $d$ and $r_d$ as the error. If the errors $r_d$ are sufficiently small, then the moments will be close to those of a normal distribution. We define approximations of the mean and variance of $g(n)$ by $$\mu_\p(g) = \sum_{p \in \p}g(p)\frac{h(p)}{p} \text{ and }\sigma_\p(g)^2= \sum_{p \in \p} g(p)^2\frac{h(p)}{p}\left (1- \frac{h(p)}{p} \right ).$$

We then have the following:

\begin{theorem}[Theorem 4 in \cite{GS}]\label{KSthm} For $k \in \N$, let $\mathcal{D}_k(\p)$ be the set of squarefree products of $k$ or fewer primes in $\p$ that are contained in $\A$. Define $C_k  = \dfrac{\Gamma(k+1)}{2^\frac{k}{2}\Gamma(\frac{k}{2} +1) }$. Then, uniformly for all even positive integers $k \le \sigma_\p(g)^\frac{2}{3}$ we have
 
$$\sum_{n \in \A}\left ( g(n) - \mu_\p(g)\right )^k = C_kN \sigma_\p(g)^k \left (1 + O \left (\frac{k^3}{\sigma_\p(g)^2} \right ) \right ) + O\left ( \left ( \sum_{p \in \p} \frac{h(p)}{p} \right )^k\sum_{d \in \mathcal{D}_k(\p) } |r_d| \right )$$ and uniformly for all odd positive integers $k \le \sigma_\p(g)^\frac{2}{3}$, $\sum_{n \in \A}\left ( g(n) - \mu_\p(g) \right )^k $ satisfies \begin{equation}\label{oddk}\sum_{n \in \A}\left ( g(n) - \mu_\p(g)\right )^k   \ll C_kN\sigma_\p(g)^k\frac{k^\frac{3}{2} }{\sigma_\p(g)} +  \left ( \sum_{p \in \p} \frac{h(p)}{p} \right )^k\sum_{d \in \mathcal{D}_k(\p)} |r_d|.\end{equation} 
\end{theorem}

Theorem \ref{KSthm} can be used in the following way. Suppose $\A$ is an infinite sequence and $g:\A \rightarrow \R$ an additive function with $|g(p)| \le 1$ for every prime $p$. For $N \in \N$, we may define $\A(N)$ as $a_1, a_2, \ldots a_N$ and $\p(N) = \{ p \text{ prime } : p \le Y(N) \}$ for some appropriately chosen function $Y(N)$. We then define $g_N(n)$ as $$g_{N}(n) = \sum_{\genfrac{}{}{0pt}{}{p \in \p(N)}{ p \mid n}} g(p)$$ and apply Theorem \ref{KSthm} to $g_N$ on $\A(N)$. For notational purposes, let $\mu_N := \mu_{\p(N)}(g_N)$ and $\sigma_N := \sigma_{\p(N)}(g_N)$. If $Y(N)$ is chosen appropriately so that 
\begin{enumerate}[(i)]
\item the errors $\frac{k^3}{\sigma_N^2}$ and $\frac{k^\frac{3}{2} }{\sigma_N}$ both tend to $0$ and $N \rightarrow \infty$,
\item $$\frac{\left ( \sum_{p \in \p(N)} \frac{h(p)}{p} \right )^k\sum_{d \in \mathcal{D}_k(\p(N)) } |r_d|}{\sigma_N^k} = o(N), \text{ and}$$
\item $g_N(n) = g(n) + o(\sigma_N)$,
 \end{enumerate}

then the moments of $\frac{g(n) - \mu_N}{\sigma_N}$ tend to $C_k$ for $k$ even and to zero for $k$ odd as $N \rightarrow \infty$.

Suppose that we are in the special case where $\A$ is an increasing sequence and let  $\mu_g(N) = \sum_{p \le a_N}g(p)\frac{h(p)}{p}$ and $\sigma^2_g(N) = \sum_{p \le a_N}g(p)^2\frac{h(p)}{p} \left ( 1 - \frac{h(p)}{p} \right )$. If we additionally have that $\sigma_g(N) \rightarrow \infty$ and that $$\sigma_g(N) = \sigma_N + o(\sigma_N) \text{ and } \mu_g(N) = \mu_N + o(\sigma_N),$$ then we actually get that the moments of $\frac{g(n) - \mu_g(N)}{\sigma_g(N)}$ tend to $C_k$ for $k$ even and $0$ for $k$ odd as $N \rightarrow \infty$. Since the $k^{th}$ moments of the standard normal distribution $\mathcal{N}(0,1)$ are $0$ for $k$ odd and $C_k$ for $k$ even and the normal distribution $\mathcal{N}(0,1)$ is determined by its moments, we  then have that $$\displaystyle {v_x \left ( n; \frac{g(n) - \mu_g(x)}{\sigma_g(x)} \le z \right ) = \frac{|\{n \in \A, n \le x: \frac{g(n) - \mu_g(x)}{ \sigma_g(x)}\le z \}|}{|\{n \in \A, n \le x \} |}}$$ converges weakly to $$G(z) = \frac{1}{\sqrt{2\pi}} \int_{-\infty}^{z} \! e^\frac{w^2}{2} \, \mathrm{d}w.$$

This will all be the case for the set of squarefree numbers which we examine in the next section.

\subsection{Application to Squarefree Numbers}\label{pf}

Let $S$ be the sequence of squarefree natural numbers and suppose that $g:S \rightarrow \R$ is an additive function with $|g(p)| \le 1$ for every prime $p$. As $S$ is increasing and $a_N = O(N)$, rather than considering the first $N$ elements of $S$, we will instead work with $S(X) = \{ n \in S : n \le X)$ and adjust the notation accordingly. Define $h(d) =  \frac{d}{\sigma(d)}$ where $\sigma(d)$  is the sum of the divisors of $d$. The function $h(d)$ is multiplicative and we have the following:

\begin{proposition}\label{bdr}
For $d \in S(X)$, $r_d = O(X^\frac{3}{4}).$
\end{proposition}
\begin{proof}
The basic idea here is that if $d$ is squarefree, then the squarefree numbers should distribute approximately uniformly within the set of squarefree classes modulo $d^2$. Let $a$ be a squarefree integer with $0<a<d^2$. By Corollary 1 to Theorem 1 in \cite{CohenRobinson}, $$\left | \left \{ n \in S(X): n \equiv a \pmod{d^2} \right \} \right | = \dfrac{6}{\pi^2}\dfrac{1}{\displaystyle{  \prod_{p \mid d} (p^2 -1)} }X + O\left (\sqrt{X} \right ).$$ There are $\phi(d)$ squarefree classes modulo $d^2$ that are congruent to zero modulo $d$. This yields \begin{equation}\label{divd}\left | \left \{ n \in S(X): d \mid n  \right \} \right | = \dfrac{6}{\pi^2}\dfrac{1}{\displaystyle{  \prod_{p \mid d} (p+1)} }X + O\left (d\sqrt{X} \right ) =  \dfrac{6}{\pi^2}\frac{1}{\sigma(d)}X +  O\left (d\sqrt{X} \right ).\end{equation} 

As $|S(X)| = \dfrac{6}{\pi^2}X +   O\left (\sqrt{X} \right )$ and  $\frac{h(d)}{d} = \frac{1}{\sigma(d)}$, this proves the result for $d = O(X^\frac{1}{4})$. However, for $d = \Omega(X^\frac{1}{4})$, both $|S(X)_d| = O(N^\frac{3}{4})$ and $\frac{h(d)}{d}X = O(X^\frac{3}{4})$, so it follows that $r_d = O(X^\frac{3}{4})$ for such $d$ as well.
\end{proof}

We can now prove Theorem \ref{EKSF}.

\begin{proof}[Proof of Theorem \ref{EKSF}]


Set $Y(X) = X^{\frac{1}{B(X)^\frac{2}{3}}}$ and define $$\mu_X = \sum_{p \le Y(X)}g(p)\frac{h(p)}{p} \text{ and }\sigma_X^2 = \sum_{p \le Y(X)}g(p)^2\frac{h(p)}{p}\left (1- \frac{h(p)}{p} \right ).$$ As $\frac{h(p)}{p} = \frac{1}{p+1}$, we then get that $$\mu_X = \sum_{p \le Y(X)}\frac{g(p)}{p+1} = \sum_{p \le Y(X)}\frac{g(p)}{p} + O(1) = A(Y(X)) + O(1)$$ and $$\sigma_X^2 = \sum_{p \le Y(X)}\frac{g(p)^2}{p} + O(1) = B(Y(X))^2 + O(1).$$ As we assume that $B(X) \rightarrow \infty$ as $X \rightarrow \infty$, it then follows that $Y(X) \rightarrow \infty$ as $X \rightarrow \infty$ and therefore that $B(Y(X)) \rightarrow \infty$ and $\sigma_X\rightarrow \infty$ as $X \rightarrow \infty$. We therefore get that the error terms $\frac{k^3}{\sigma_X^2}$ and $\frac{k^\frac{3}{2} }{\sigma_X}$ both tend to $0$ as $X \rightarrow \infty$ for fixed $k$.

Next, consider $$\frac{\left ( \sum_{p \le Y(X)} \frac{h(p)}{p} \right )^k\sum_{d \in \mathcal{D}_k(Y(X)) } |r_d|}{\sigma_X^k},$$ where $\mathcal{D}_k(Y(X))$  is the set of products of $k$ or fewer primes less than or equal to $Y(X)$ contained in $S(X)$. Trivial estimates show that $|\mathcal{D}_k(Y(X))| = O\left ( kX^\frac{k}{B(X)^\frac{2}{3}}\right )$. As $B(X) \rightarrow \infty$, for fixed $k$ we therefore have that $|\mathcal{D}_k(Y(X))| = o(X^\epsilon)$ for any $\epsilon \ge 0$. Similarly, Mertens' Theorem shows that $$\left ( \sum_{p \le Y(X) } \frac{1}{p+1} \right )^k  = (\log \log Y(X) + O(1))^k$$ and therefore that $$\left ( \sum_{p \le Y(X) } \frac{1}{p+1} \right )^k = o(X^\epsilon)$$ for any $\epsilon \ge 0$ as well. Combined with Proposition \ref{bdr}, we then get that  $$\frac{\left ( \sum_{p \le Y(X)} \frac{h(p)}{p} \right )^k\sum_{d \in \mathcal{D}_k(Y(X)) } |r_d|}{\sigma_X^k} = O(X^{\frac{3}{4} + \epsilon})$$ for any $\epsilon > 0$.

Since a number $n \le X$ can have at most ${B(X)}^\frac{2}{3}$ prime factors greater than $X^{\frac{1}{B(X)^\frac{2}{3}}}$, it therefore follows $$g(n) - g_X(n) = g(n)-\sum_{\genfrac{}{}{0pt}{}{p \le Y(X)}{p \mid n}} g(p) =\sum_{\genfrac{}{}{0pt}{}{Y(X) < p \le X}{p \mid n}} g(p)  \le {B(X)}^\frac{2}{3}.$$ In order for $g(n) - g_X(n) = o(\sigma_X)$, it therefore suffices for $B(X) = \sigma_X + o(\sigma_X)$. Along with the discussion at the end of Section \ref{sEK}, it is then enough to show that $B(X) = \sigma_X +  o(\sigma_X)$ and that $A(X) = \mu_X + o(\sigma_X)$ to complete the proof.

Recalling that $\mu_X = A(Y(X)) + O(1)$ and $\sigma_X = B(Y(X)) + O(1)$, we have $$A(X) - A(Y(X)) \le \sum_{X^{\frac{1}{B(X)^\frac{2}{3}}} \le p \le X} \frac{1}{p} = \log \log X - \log \log X^{\frac{1}{B(X)^\frac{2}{3}}}  + O(1) = O(\log \log \log X),$$ with the first equality coming from Merten's theorem and the final equality following from the fact that $B(X) = O(\sqrt{\log \log X})$.

Similarly, for sufficiently large $X$, we have $$B(X) - B(Y(X)) \le  B(X)^2 - B(Y(X))^2 \le \sum_{X^{\frac{1}{B(X)^\frac{2}{3}}} \le p \le X} \frac{1}{p} =  O(\log \log \log X).$$ The fact that $B(X) = O(\sqrt{\log \log X})$ and the assumption that $B(X) = \omega \left( \log \log\log(x) \right )$ then give that $B(X) = \sigma_X +  o(\sigma_X)$ and $A(X) = \mu_X + o(\sigma_X)$.
\end{proof}

\end{document}